\documentclass{article}
\usepackage[utf8]{inputenc}
\usepackage[T1]{fontenc}
\usepackage{amsmath, amsthm, amssymb}
\usepackage{scalerel}
\usepackage{url}
\usepackage[numbers]{natbib}

\usepackage{hyperref}
\usepackage{mathtools}
\usepackage{geometry}
\usepackage{wasysym}
\usepackage{fancyhdr}
\usepackage{fontawesome}
\usepackage{isotope}
\usepackage{verbatim}
\usepackage{graphicx}

\makeatletter
\let\origps@plain\ps@plain
\newcommand\MakePlainPagestyleEmpty{\let\ps@plain\ps@empty}
\newcommand\MakePlainPagestylePlain{\let\ps@plain\origps@plain}
\makeatother
\usepackage{blindtext}

\newtheorem{dfn}{Definition}[section]
\newtheorem{prop}{Proposition}[section]
\newtheorem{thm}{Theorem}[section]
\newtheorem{cor}{Corollary}[section]
\newtheorem{lem}{Lemma}[section]
\newtheorem{rmk}{Remark}[section]
\newtheorem{hyp}{Hypothesis}[section]
\numberwithin{equation}{section}

\newcommand{\bca}{\begin{cases}}
\newcommand{\eca}{\end{cases}}

\newcommand{\lb}{\left(}

\newcommand{\rb}{\right)}
\newcommand{\lmb}{\left[}
\newcommand{\rmb}{\right]}
\newcommand{\Ha}{{\mathcal{H}}}

\newcommand{\Ca}{{\mathcal{C}}}

\newcommand{\R}{\mathbb{R}}

\newcommand{\gd}{\nabla}
\newcommand{\rta}{\rightarrow}
\newcommand{\lw}{\left|}
\newcommand{\rw}{\right|}
\newcommand{\be}{\begin{equation}}
\newcommand{\ee}{\end{equation}}

\newcommand{\bt}{\begin{thm}}
\newcommand{\et}{\end{thm}}
\newcommand{\bc}{\begin{cor}}
\newcommand{\ec}{\end{cor}}
\newcommand{\bl}{\begin{lem}}
\newcommand{\el}{\end{lem}}
\newcommand{\norm}[1]{\left\lVert#1\right\rVert}
\newcommand{\normm}[1]{{\left\vert\kern-0.25ex\left\vert\kern-0.25ex\left\vert #1 
    \right\vert\kern-0.25ex\right\vert\kern-0.25ex\right\vert}}

\def\Xint#1{\mathchoice
{\XXint\displaystyle\textstyle{#1}}%
{\XXint\textstyle\scriptstyle{#1}}%
{\XXint\scriptstyle\scriptscriptstyle{#1}}%
{\XXint\scriptscriptstyle\scriptscriptstyle{#1}}%
\!\int}
\def\XXint#1#2#3{{\setbox0=\hbox{$#1{#2#3}{\int}$ }
\vcenter{\hbox{$#2#3$ }}\kern-.6\wd0}}

\def\dashint{\Xint-}

\usepackage{graphicx}
\usepackage{caption}
\usepackage{subcaption}

\newcommand\restr[2]{{
  \left.\kern-\nulldelimiterspace 
  #1 
  \vphantom{\big|} 
  \right|_{#2} 
  }}

\title{Is Mean Curvature Flow a Gradient Flow?}
\author{Zhonggan Huang\thanks{The author was partially supported by W. Feldman's NSF grant DMS-2009286.}\\
Department of Mathematics,\;University of Utah}

\begin{document}

\setcounter{page}{1}

\maketitle
\begin{abstract}
   It is well-known that the mean curvature flow is a formal gradient flow of the perimeter functional. However, by the work of Michor and Mumford \citep{Michor2,Michor1}, the formal Riemannian structure that is compatible with the gradient flow structure induces a degenerate metric on the space of hypersurfaces. It is then natural to ask whether there is a nondegenerate metric space of hypersurfaces, on which the mean curvature flow admits a gradient flow structure. In this paper we study the mean curvature flow on two nondegenerate metric spaces of simple closed plane curves: the uniformness-preserving metric structure proposed by Shi and Vorotnikov \citep{unfmcf} and the curvature-weighted structure proposed by Michor and Mumford \citep{Michor1}, and prove that the mean curvature flow is not a gradient flow in either of the spaces.
\end{abstract}

\section{Introduction}
\emph{Mean Curvature Flow} (MCF) is widely known to be a formal gradient flow of the perimeter functional. To be more specific, let \(\Gamma_t\coloneqq J_t(\Gamma_0)\subset\R^d\) be a smooth family of hypersurfaces, where \(J_t\) satisfies the following ODE
\[
\bca
\displaystyle\dot{J}_t=\Vec{V}(t,J_t),&\,t\in [0,1],\\
\displaystyle J_0(x)=x.&
\eca
\]
If the vector field \(\vec{V}(t,\cdot)=\vec{H}=\kappa \vec{N}\) is the mean curvature vector field on \(\Gamma_t\), then the trajectory \(\Gamma_t\) is called a mean curvature flow. On the other hand, by the first variation formula of the perimeter, a MCF satisfies
\be
    \frac{d}{dt} \Ha^{d-1}(\Gamma_t)  = -\int_{\Gamma_t} |\vec{H}|^2 d\Ha^{d-1},
\ee
where \(\Ha^{d-1}\) is the \(d-1\)-Hausdorff measure on \(\R^d\). Formally speaking, we can naturally define a Riemannian structure on the space of hypersurfaces: at a hypersurface \(\Gamma\subset\R^d\), the tangent space is defined as all the normal vector fields on \(\Gamma\), and for \(\vec{V}^\bot,\vec{W}^\bot\) in the tangent space, there is an inner product
\be
(\vec{V}^\bot,\vec{W}^\bot)_{\Gamma}\coloneqq \int_{\Gamma} \vec{V}^\bot\cdot\vec{W}^\bot d\Ha^{d-1}.\label{wronginner}
\ee
In the above Riemannian structure a MCF is indeed a formal gradient flow of perimeter \(\Ha^{d-1}(\cdot)\).

This structure, however, turns out to be degenerate. By the work of Michor and Mumford \citep{Michor2,Michor1}, the geodesic distance between any two hypersurfaces is zero. Now it is very natural to ask: is there an alternative gradient flow structure for MCF on a nondegenerate metric space of hypersurfaces?

To deal with the issue that pushforwards of Hausdorff measures do not preserve uniformness under the flow of general normal vector fields, Shi and Vorotnikov \citep{unfmcf} proposed a new Riemannian structure with tangent space defined as a subspace of vector fields on \(\Gamma\) composed of \(\vec{V}\) such that
\be
\text{div}_\Gamma \vec{V} \equiv Const.,\label{coherenceoftangent}
\ee
where \(\text{div}_\Gamma\) is defined as the tangential divergence on \(\Gamma\). For \(\vec{V},\vec{W}\) in the new tangent space, the inner product is defined as
\be
(\vec{V},\vec{W})_U\coloneqq \int_{\Gamma} \vec{V}\cdot\vec{W} d\Ha^{d-1}.\label{rightinner}
\ee
Notice that, different from \eqref{wronginner}, the right-hand side of \eqref{rightinner} involves the tangential components of \(\vec{V}\) and \(\vec{W}\).

In the Riemannian structure described above, Shi and Vorotnikov discussed the gradient flow of perimeter functional and discovered a new geometric flow called \emph{Uniformly Compressing Mean Curvature Flow} (UCMCF). The MCF itself can also be understood as a flow in the new structure by modifying the tangential components of the mean curvature vector fields \citep{tangenti}, but because of the modifications in \eqref{rightinner}, the velocity field that drives UCMCF has a nontrivial difference from that of MCF in general. 

Moreover, UCMCF is also a gradient flow of \(\log\) perimeter in the space of normalized Hausdorff measures. This space is canonically embedded into the Wasserstein space (in the sense of Otto's formal Riemannian structure \citep{Otto3}) because the tangent space is defined as before and the inner product is simply a normalized version of the inner product \eqref{rightinner}
\be
(\vec{V},\vec{W})_{N}\coloneqq \dashint_{\Gamma} \vec{V}\cdot\vec{W} d\Ha^{d-1}.\label{rightinner2}
\ee
This embedding ensures that the geodesic distance induced by \eqref{rightinner2} is nondegenerate because it is lower bounded by the Wasserstein distance. Similarly, tangentially modified MCF can also be discussed in this normalized structure. More detailed discussions of the structure and the flows are contained in Section 2.

In \citep{Michor1}, Michor and Mumford also proposed a nondegenerate metric for plane curves by adding a term that involves curvatures to \eqref{wronginner} 
\be
(\vec{V}^\bot,\vec{W}^\bot)_{M}\coloneqq \int_{\Gamma}\lb1+\kappa^2\rb \vec{V}^\bot\cdot\vec{W}^\bot d{\Ha^{1}},\label{michorrightinner}
\ee
where \(\kappa\) is the scalar curvature of the curve. In this structure, MCF can be understood as a flow directly without any further modifications.

The Riemannian structures \eqref{rightinner2} and \eqref{michorrightinner} are appealing for various reasons. First and foremost they both result in nondegenerate geodesic distances on the space of embedded curves. The Shi-Vorotnikov uniformness-preserving metric is also appealing for preserving the ``density of grid points'' on the curve, a property which has been observed to be useful in applications \citep{app}. Both metrics were introduced in the context of studying the MCF or MCF-like flows as gradient flows.  It is natural to ask if the MCF itself is a gradient flow of some functional under these metric structures.  In this paper we show that it is not in both cases:
\begin{thm}
  The curve-shortening flow is not a gradient flow either in the Riemannian structure \eqref{rightinner2} proposed by Shi and Vorotnikov \citep{unfmcf} or \eqref{michorrightinner} proposed by Michor and Mumford \citep{Michor1}.\label{1}\label{2}
\end{thm}

As far as the author knows, no results concerning non-gradient-flow properties for MCF have been published (neither do those on rigorous gradient flow structures). On the other hand, many works over the past few decades indicate that MCF can be well approximated by true gradient flows. 
For example, it is known that MCF can be understood as the sharp interface limit of the Allen-Cahn equation, which is an \(L^2\) gradient flow of a Ginzberg-Landau type functional \citep{Ilmanen1}. The well-known MBO thresholding scheme \citep{MERRIMAN1994334} for MCF is proved to have a discrete gradient flow structure \citep{Otto2,Otto1}.

\section{Some Preparations}
\subsection{On a Submanifold of Wasserstein Space}

In \citep{unfmcf}, Shi and Vorotnikov discussed a special Riemannian structure on the space
\[
\Ca_d\coloneqq \left\{\Gamma\in C^\infty(\mathbb{S}^{d-1};\R^d)\,;\,\Gamma \text{ is an embedding of }\mathbb{S}^{d-1} \text{ in }\R^d\right\}/\sim,
\]
where ``\(\sim\)'' is the equivalence relation that identifies any two embeddings with the same image. Although the arguments in this section can also be applied to higher codimensional cases, we will focus on hypersurfaces. The new structure naturally gives an embedding \(i:  \Ca_d\hookrightarrow \mathbb{W}_2(\R^d)\) to the 2-Wasserstein space of probability measures on \(\R^d\), where the assignment \(i\) is defined as
\be
i([\Gamma])\coloneqq\frac{\restr{\mathcal{H}^{d-1}}{\text{Im}(\Gamma)}}{\mathcal{H}^{d-1}(\text{Im}(\Gamma))}
\ee
We will not distinguish \(\Gamma\), \([\Gamma]\), \(i([\Gamma])\) and \(\text{Im}(\Gamma)\) when it is unambiguous. For example, ``\(d\Gamma\)'' will simply mean ``\(d \lb i([\Gamma])\rb\)''.

 \begin{dfn}[\textbf{Coherent Space}]
  We call \(\Ca_d\) endowed with the metric \eqref{rightinner2} the \textbf{Coherent Space} of hypersurfaces in \(\R^d\). At each \(\Gamma\in \Ca_d\), we would call \(\mathbb{T}_\Gamma \Ca_d\) the space of all vector fields \(\vec{V}\) on \(\Gamma\) satisfying \eqref{coherenceoftangent} the \textbf{Coherent Tangent Space} at \(\Gamma\).
  \end{dfn}
  
We start by discussing the paths in \(\mathcal{C}_d\). Let \(\Phi_t(x)\) be a flow map of the following ODE
\be
\bca
\displaystyle\dot{\Phi}_t=\Vec{V}(t,\Phi_t),&\,t\in [0,1],\\
\displaystyle\Phi_0(x)=x.&
\eca
\ee
We consider $\Gamma_t \coloneqq \Phi_t (\Gamma_0)$ (flow of images) and observe according to first variation formula \citep{Ilmanen1998LecturesOM}, for every smooth test function $\zeta$, \(\Gamma_t\) (as normalized Hausdorff measures) should satisfy
\be
\begin{split}
   \frac{d}{dt} \int \zeta d \Gamma_t & = 
  \frac{d}{dt} \int_{\Phi_t (\Gamma_0)} \zeta d \frac{\Ha^{d - 1}
  }{\Ha^{d - 1}  (\Gamma_t)} \\
  & =  \frac{1 }{\Ha^{d - 1}  (\Gamma_t)}\int_{\Gamma_t} \nabla \zeta \cdot \vec{V} (t, \cdot) d \Ha^{d-1} + \frac{1}{\Ha^{d - 1}  (\Gamma_t)}\int_{\Gamma_t} \zeta \text{div}_{\Gamma_t}
  \vec{V} (t, \cdot) d \Ha^{d - 1}  \\
  &  \quad - \frac{1
  }{\Ha^{d - 1}  (\Gamma_t)} \int_{\Gamma_t} \text{div}_{\Gamma_t} \vec{V} (t, \cdot) d\Ha^{d - 1} \cdot \frac{1
  }{\Ha^{d - 1}  (\Gamma_t)}\int_{\Gamma_t} \zeta d \Ha^{d - 1} \\
  & =  \int \nabla \zeta \cdot \vec{V} (t, \cdot) d \Gamma_t + \int \zeta \left(
  \text{div}_{\Gamma_t} \vec{V} (t, \cdot) - \int \text{div}_{\Gamma_t} \vec{V} (t, \cdot)  d
  \Gamma_t \right) d \Gamma_t .  \label{peripush}
\end{split}
  \ee
Observe that if \(\vec{V}\) is smooth, then we have by Poincar\'e inequality,
\be
\begin{split}
    \left| \frac{d}{dt} \int \zeta d \Gamma_t \right|  \leqslant 
  C \left( \norm{\vec{V}} _{L_{\Gamma_t}^2} + \norm{ \text{div}_{\Gamma_t} \vec{V} (t, \cdot) - \int \text{div}_{\Gamma_t} \vec{V} (t, \cdot)  d
  \Gamma_t}_{L_{\Gamma_t}^2} \right) \| \nabla \zeta \|_{L_{\Gamma_t}^2}, 
  \label{estimate}
\end{split}
 \ee
which means that the path \(\Gamma_t\) is in fact a Lipschitz path in Wasserstein space. Let us now interpret \(\Phi_t(\Gamma_0)\) (flow of images) as pushforward of measures.

\begin{prop}
  Let $\Gamma_t = \Phi_t (\Gamma_0)$ be defined as before with respect to a
  smooth vector field $\vec{V}$, then $\Gamma_t$ as a family of probability measures should satisfy for all smooth test functions $\zeta$,
  \[ \frac{d}{dt} \int \zeta d \Gamma_t = \int \nabla \zeta \cdot
     (\vec{V }^{\bot} + \nabla_{\Gamma_t} U_t) d \Gamma_t, \]
  where $\vec{V }^{\bot}$ is the normal component of $\vec{V}$ and
  $U_t$ satisfies the following elliptic equation
  \begin{equation}
  \bca
  \displaystyle- \Delta_{\Gamma_t} U_t = \int \vec{V} \cdot \vec{H} d \Gamma_t -\vec{V} \cdot \vec{H},\\
  \displaystyle\int U_t d\Gamma_t=0.
  \eca
  \label{theelliptic}
\end{equation}
Moreover, if $\overline{\Phi
  }_t$ is the flow map of the vector field $\vec{V }^{\bot} +
  \nabla_{\Gamma_t} U_t$, then we have
  \be\frac{\restr{\Ha^{d-1}}{\Gamma_t}}{\Ha^{d-1}(\Gamma_t)} = \overline{\Phi }_t \# \frac{\restr{\Ha^{d-1}}{\Gamma_0}}{\Ha^{d-1}(\Gamma_0)} . \label{coherence}\ee
\end{prop}

\begin{rmk}
   \begin{itemize}
    \item[(i)] The transformation
    \[
    \begin{split}
        P = P_{\Gamma_t} :\,\mathbb{T}_{\Gamma_t}\mathbb{W}_2&\rta \mathbb{T}_{\Gamma_t}\mathbb{W}_2\\
         \vec{V}\quad &\mapsto\vec{V }^{\bot} + \nabla_{\Gamma_t} U_t
    \end{split}  \]
    is a projection, where \(\mathbb{T}_{\Gamma_t}\mathbb{W}_2\) is the space of all smooth vector fields on \(\Gamma_t\) (we would call this the \emph{Wasserstein tangent space}). Indeed, if we call \(\emph{ndiv}_\Gamma\vec{W}\coloneqq\emph{div}_\Gamma\vec{W}-\int\emph{div}_\Gamma\vec{W} d\Gamma\), then it
    can be seen by the following computation
    \be
    \begin{split}
        \int \zeta \emph{ndiv}_{\Gamma_t} (\vec{V }^{\bot} +
      \nabla_{\Gamma_t} U_t) d \Gamma_t & =  \int \zeta \emph{div}_{\Gamma_t}
      \vec{V }^{\bot} d \Gamma_t - \int \emph{div}_{\Gamma_t}
      \vec{V }^{\bot} d \Gamma_t \cdot \int \zeta d \Gamma_t \\
      &\quad+ \int
      \zeta \Delta_{\Gamma_t} U_t d \Gamma_t\\
      & =  \int \zeta \left( - \vec{V} \cdot \vec{H} + \int \vec{V} \cdot
      \vec{H} d \Gamma_t \right) d \Gamma_t\\
      &\quad- \int \zeta \left( \int \vec{V}
      \cdot \vec{H} d \Gamma_t - \vec{V} \cdot \vec{H} \right) d \Gamma_t\\
      & =  0.
    \end{split}
        \ee
    \item[(ii)] Observe that the projection $P=P_{\Gamma_t}$ only relies on the information of
    the current surface $\Gamma_t$. Moreover, the image of the projection \(P_{\Gamma}\) is indeed the coherent tangent space because of the above remark. Since the structure defined in \eqref{rightinner2} is simply the restriction of Otto's formal Riemannian structure \citep{Otto3}, we arrive at the conclusion that \(\Ca_d\) is indeed a Riemannian submanifold of \(\mathbb{W}_2(\R^d)\).

    \item[(iii)] The vector field $\vec{V }^{\bot} + \nabla_{\Gamma_t}
    U_t$ can be written as the gradient of some function. Indeed, if
    we take an $\varepsilon$-tube neighborhood of $\Gamma_t$, then we may just
    find that $\vec{V }^{\bot} = \nabla V_t$ on \(\Gamma_t\) for some function
    $V_t$ that is nonconstant only along the normal trajectories. Extending
    $U_t$ as a constant along the normal trajectories, we then can write, at
    least in a $\varepsilon$-tube neighborhood of $\Gamma_t$ there is a function \(V_t+U_t\) such that
    $\vec{V }^{\bot} + \nabla_{\Gamma_t} U_t = \nabla (V_t + U_t)$ on \(\Gamma_t\).
    \item[(iv)] This proposition tells us that given any smooth path of curves \(\Gamma_t,\,t\in[0,1]\), there is a smooth reparametrization \(\bar{\Phi}_t(\theta):\mathbb{S}^{d-1}\rta \Gamma_t\subset\R^d\) satisfying
    \[\sqrt{\det\lb D_\theta\bar{\Phi}_t^TD_\theta\bar{\Phi}_t\rb}\equiv l_t/|\mathbb{S}^{d-1}|,\]
    where \(l_t\coloneqq \Ha^{d-1}(\Gamma_t)\). In particular if \(d=2\) then we have \(|\partial_\theta \bar{\Phi}_t|\equiv l_t/2\pi\).
    \item[(v)] When \(d=2\), we know that \(\Ca_2\) is connected by paths of above type. This is indicated by the Whitney–Graustein theorem, which states that regular homotopy classes of plane curves can be classified by their turning numbers.
  \end{itemize}
\label{remark1}

\end{rmk}

\begin{proof}
First observe that 
\be
\begin{split}
    \int \zeta \text{div}_{\Gamma_t} \vec{V} d \Gamma_t & =  \int
  \text{div}_{\Gamma_t} \lb\zeta \vec{V}\rb - \nabla_{\Gamma_t} \zeta \cdot \vec{V}
  d \Gamma_t\\
  & =  - \int \zeta \vec{V} \cdot \vec{H} + \nabla_{\Gamma_t} \zeta \cdot
  \vec{V} d \Gamma_t.
\end{split}
\ee
Now the last term in formula \eqref{peripush} becomes
\begin{eqnarray*}
  \int \nabla \zeta \cdot \vec{V} d \Gamma_t + \int \zeta
  \text{ndiv}_{\Gamma_t} \vec{V} d \Gamma_t & = & \int \zeta \left( \int
  \vec{V} \cdot \vec{H} d \Gamma_t - \vec{V} \cdot \vec{H} \right) +
  \nabla^{\bot} \zeta \cdot \vec{V} d \Gamma_t .
\end{eqnarray*}
Using this equation and the estimate \eqref{estimate}, we are able to solve
\[ \cfrac{d}{dt} \int \zeta d \Gamma_t = \int \nabla \zeta \cdot
   \vec{G}_t d \Gamma_t = \int \nabla \zeta \cdot \vec{V} d
   \Gamma_t + \int \zeta \text{ndiv}_{\Gamma_t} \vec{V} d \Gamma_t, \]
with $\vec{G}_t$ minimizing the \(L^2\)-energy $\int | \vec{G}_t
|^2 d \Gamma_t$. This is equivalent to solve
\begin{eqnarray*}
  \int \nabla \zeta \cdot \vec{G}_t d \Gamma_t & = & \int \nabla
  \zeta \cdot \vec{V} d \Gamma_t + \int \zeta \text{ndiv}_{\Gamma_t} \vec{V} d
  \Gamma_t\\
  & = & \int \zeta \left( \int \vec{V} \cdot \vec{H} d \Gamma_t - \vec{V}
  \cdot \vec{H} \right) + \nabla^{\bot} \zeta \cdot \vec{V} d \Gamma_t .
\end{eqnarray*}
Fixing $\zeta \equiv 0$ on $\Gamma_t$, we observe that a solution
$\vec{G}_t$ should share the same normal component with $\vec{V} $.
Hence we need only compute the tangential component, which is equivalently
solving
\begin{equation}
  \int \nabla_{\Gamma_t} \zeta \cdot \vec{Q}_t d \Gamma_t = \int
  \zeta \left( \int \vec{V} \cdot \vec{H} d \Gamma_t - \vec{V} \cdot \vec{H}
  \right) d \Gamma_t .\label{tangential}
\end{equation}
Moreover, because the normal part of $\vec{G}_t$ is fixed, we just need to minimize the \(L^2\)-energy of its tangential part
$\vec{Q}_t$. To that end, we recall that the space of tangential vector fields on \(\Gamma_t\) (which has the same topology as the unit sphere) can be decomposed as the direct sum of two orthogonal subspaces: gradients of functions and divergence-free vector fields. Here a tangential vector field $\vec{S}$ is called \emph{divergence-free} if for all test function $\zeta\in C_0^\infty(\R^d)$
\begin{equation}
  \int \nabla_{\Gamma_t} \zeta \cdot \vec{S} d \Gamma_t = 0.\label{divergencefree}
\end{equation}
Observe that if \(\vec{Q}_t\) is a solution to \eqref{tangential} and \(\vec{S}\) is a divergence-free vector field, then $\vec{Q}_t + \vec{S}$ is still a
solution to \eqref{tangential}. Therefore, to minimize the \(L^2\)-energy of the tangential component it is equivalent to find \(\vec{Q}_t^\ast\) satisfying
\[ \int | \vec{Q}_t^\ast + \vec{S} |^2 d \Gamma_t\ge \int | \vec{Q}_t^\ast |^2 d \Gamma_t,\,\forall \text{ divergence-free }\vec{S}, \]
which gives us the following variational formula
\[ \int \vec{Q}_t^\ast \cdot \vec{S} d \Gamma_t = 0,\,\forall \text{ divergence-free }\vec{S}. \]
Now we may write for some
function $U_t\in H^1(\Gamma_t)$
\[ \vec{Q}_t^\ast = \nabla_{\Gamma_t} U_t , \]
and then we can modify \eqref{tangential} as
\[ \int \nabla_{\Gamma_t} \zeta \cdot \nabla_{\Gamma_t} U_t d \Gamma_t = \int
   \zeta \left( \int \vec{V} \cdot \vec{H} d \Gamma_t - \vec{V} \cdot \vec{H}
   \right) d \Gamma_t, \]
which is equivalently solving the elliptic equation \eqref{theelliptic}. We immediately obtain the smoothness of \(\vec{Q}_t^\ast\) from the standard elliptic theory. Equality \eqref{coherence} is obtained by observing that \(\overline{\Phi}_t\) is a constant-speed reparametrization for \(\Gamma_t\) for all \(t\).
\end{proof}

\subsection{The Uniformly Compressing MCF}

Let us now discuss a geometric flow that is very similar to MCF. We define the
{\emph{log perimeter functional}} $\mathcal{R}$ on \(\mathcal{C}_d\)
as $\mathcal{R} (\Gamma) \coloneqq \log \Ha^{d - 1}  (\Gamma)$. Its differential can be written
explicitly: Let $\Gamma_t$ be a path with respect to some vector field
$\vec{V}$ for $t \in [0, 1]$. Then we have by the first variation formula
\begin{eqnarray*}
  \left. \frac{d}{dt} \mathcal{R} (\Gamma_t) \right|_{t = 0} & = &
  \left. \frac{d}{dt} \log \Ha^{d - 1}  (\Gamma_t) \right|_{t = 0}\\
  & = & - \int_{\Gamma_0 } \vec{H} \cdot \vec{V} d \frac{\Ha^{d - 1} }{\Ha^{d - 1} 
  (\Gamma_0 )}\\
  & = & - \int \vec{H} \cdot \vec{V} d \Gamma_0.
\end{eqnarray*}
The (coherent) tangential derivative of $\mathcal{R}$ at $\Gamma_0$ is the
realization of the above differential in the coherent tangent space $\mathbb{T}_{\Gamma_0}
\mathcal{C}_d$, which means that there is a unique $\nabla_{\mathcal{C}_d}
\mathcal{R}(\Gamma_0) \in \mathbb{T}_{\Gamma_0}
\mathcal{C}_d$ such that
\[ 
\int \nabla_{\mathcal{C}_d} \mathcal{R}(\Gamma_0) \cdot \vec{V} d \Gamma_0 = - \int
   \vec{H} \cdot \vec{V} d \Gamma_0, \,\forall\,\vec{V} \in \mathbb{T}_{\Gamma_0} \mathcal{C}_d.
\]
The gradient flow of \(\mathcal{R}\) with respect to the coherent metric \eqref{rightinner2} was first introduced by Shi and Vorotnikov \citep{unfmcf} and is called the \emph{uniformly compressing mean curvature flow}. It is the following interface evolution written in weak form: $\Gamma_t : [0, T] \rightarrow \mathcal{C}_d$ satisfies for all smooth test function $\zeta$,
\[ \frac{d}{dt} \int \zeta d \Gamma_t = - \int
     \nabla_{\mathcal{C}_d} \mathcal{R}(\Gamma_t) \cdot \nabla \zeta d \Gamma_t . 
\]
Let
us now compute $\nabla_{\mathcal{C}_d} \mathcal{R}$. Recall that for all coherent vector field \(\vec{V}\)
\begin{eqnarray*}
  \int \nabla_{\mathcal{C}_d} \mathcal{R} \cdot \vec{V} d \Gamma & = & - \int
  \vec{H} \cdot \vec{V} d \Gamma\\
  & = & - \int \vec{H} \cdot \vec{V}^{\bot} d \Gamma,
\end{eqnarray*}
while on the right-hand side, if we denote the normal part of
$\nabla_{\mathcal{C}_d} \mathcal{R}$ by $\vec{w}$, and its
tangential part by $\nabla_{\Gamma } W $, then we have
\begin{eqnarray*}
  - \int \vec{H} \cdot \vec{V}^{\bot} d \Gamma & = & \int \vec{w} \cdot
  \vec{V}^{\bot} d \Gamma + \int \nabla_{\Gamma } W \cdot \nabla_{\Gamma} U d
  \Gamma\\
  & = & \int \vec{w} \cdot \vec{V}^{\bot} d \Gamma - \int W
  \Delta_{\Gamma} U d \Gamma\\
  & = & \int \vec{w} \cdot \vec{V}^{\bot} d \Gamma - \int \vec{V}^{\bot}
  \cdot W \vec{H} d \Gamma\\
  & = & \int \vec{V}^{\bot} \cdot (\vec{w} - W \vec{H}) d \Gamma.
\end{eqnarray*}
Since this holds for all normal velocity fields $\vec{V}^\perp$, we get the pointwise equality $\vec{H} = - \vec{w} + W \vec{H}$. On the
other hand, by the coherence of the vector field \(\nabla_{\mathcal{C}_d} \mathcal{R}\) we have the equation $- \Delta_{\Gamma} W = \int \vec{w} \cdot
\vec{H} d \Gamma  - \vec{w} \cdot \vec{H}$, then we get
\begin{eqnarray}
  - \Delta_{\Gamma} W & = & \int \vec{w} \cdot \vec{H} d \Gamma  - \vec{w}
  \cdot \vec{H} \nonumber\\
  & = & \int (W - 1) | \vec{H} |^2 d \Gamma - (W- 1 ) | \vec{H} |^2 . 
  \label{elipforW}
\end{eqnarray}
At this stage, we obtain the following lemma.

\begin{lem}
  Let \(W\) be the unique solution to \eqref{elipforW} such that $\int W d \Gamma = 0$, then the coherent tangential derivative of \(\mathcal{R}\) at \(\Gamma\) takes the following form
  \[ \nabla_{\mathcal{C}_d} \mathcal{R} (\Gamma) = (W - 1) \vec{H} +
     \nabla_{\Gamma} W. \]
\end{lem}

\begin{rmk}
   The above lemma implies that the uniformly compressing mean curvature flow is generally not mean curvature flow. Indeed, the UCMCF is driven by \((1-W)\vec{H}\) in the normal direction instead of simply \(\vec{H}\), and when the surfaces are not of constant curvature, \(W\) is a non-trivial function, and hence the flows are distinct.
\end{rmk}

\subsection{MCF as a Flow on \texorpdfstring{$\mathcal{C}_d$}{Cd}}

In previous sections, we derived the gradient flow of log perimeter under the special submanifold structure induced by the embedding \(\mathcal{C}_d\overset{i}{\hookrightarrow} \mathbb{W}_2(\R^d)\). Now we would like to interpret MCF as a flow on \(\mathcal{C}_d\). We may write for some surface \(\Gamma\) and its mean curvature vector field \(\vec{H}=\vec{H}_\Gamma\)
\[
P_\Gamma(\vec{H})=\vec{H}+\gd_\Gamma \Sigma,
\]
where \(\Sigma\) satisfies
\be
\bca
\displaystyle-\Delta_\Gamma \Sigma = \int|\vec{H}|^2 d\Gamma - |\vec{H}|^2,&\\
\displaystyle\int\Sigma d\Gamma =0.&
\eca\label{sigma}
\ee
Thinking of \(P_{\cdot}(\vec{H}_{\cdot})\) as a vector field on \(\mathcal{C}_d\) we can see that a MCF \(\Gamma_t\) (viewed as normalized Hausdorff measures) satisfies for all test function \(\zeta\)
\[
\frac{d}{dt}\int\zeta d\Gamma_t = \int \gd\zeta \cdot P_{\Gamma_t}(\vec{H}) d\Gamma_t.
\]
That is to say, the MCF is the flow of the vector field \(P_{\cdot}(\vec{H}_{\cdot})\) on \(\mathcal{C}_d\). This flow may be referred to as ``tangentially modified'' MCF. It is equivalent to MCF when we look at the support of the flow.

\section{The Proof of Theorem \ref{1}}
 The proof of the first part is composed of Sections 3.1 and 3.2. The proof of the second part with respect to the structure proposed by Michor and Mumford follows the same outline and is given in Section 3.3.
 
\subsection{The Hypothesis and a Criterion}
 We are interested in the existence of an energy functional \(\mathcal{F}\) on \(\mathcal{C}_2\) such that its gradient flow in the coherent space is exactly the mean curvature flow (also known as curve-shortening flow). If the MCF is indeed a gradient flow in \(\Ca_2\), then we should have the following ``conservativity'' condition on the vector field \(P_{\cdot}(\vec{H}_{\cdot})\).

\begin{hyp}
For all closed paths \(\Gamma_t\) driven by a vector field \(\vec{V}_t=P_{\Gamma_t}(\vec{V}_t)\), 
\be
\int_0^1\int \lb \vec{H}_t\cdot\vec{V}_t + \gd_{\Gamma_t}\Sigma_t \cdot \gd_{\Gamma_t} U_t\rb d\Gamma_t dt = 0,
\ee
where \(\Sigma_t\) satisfies \eqref{sigma}, and \(U_t\) satisfies \eqref{theelliptic}. \label{Hypothesis}
\end{hyp} 
We present the proof by giving a contradiction of Hypothesis \ref{Hypothesis}. Observe that the first term satisfies
\[
\begin{split}
    \int_0^1\int  \vec{H}_t\cdot\vec{V}_t  d\Gamma_t dt&=\int_0^1  \frac{d}{dt}\mathcal{R}(\Gamma_t) dt\\
    =&\log{l_1}-\log{l_0}\\
    =&0,
\end{split}
\]
and hence we only have to compute
\be
\begin{split}
    \int_0^1\int \gd_{\Gamma_t}\Sigma_t \cdot \gd_{\Gamma_t} U_t d\Gamma_t dt&=-\int_0^1\int\Sigma_t \Delta_{\Gamma_t} U_t d\Gamma_t dt\\
    &=-\int_0^1\int\Sigma_t \vec{H}_t\cdot \vec{V}_t d\Gamma_t dt.
\end{split}\label{Goal2}
\ee
According to Remark \ref{remark1} (iv), we have a smooth (piecewise smooth in time) constant speed reparametrization \(\Phi_t(\theta):\mathbb{S}^1\rta\Gamma_t\) such that \(|\partial_\theta \Phi_t|\equiv l_t/2\pi\). This implies that for every \(t\in[0,1]\) we have 
\be
\int \Sigma_t \vec{H}_t\cdot \vec{V}_t d\Gamma_t = \frac{1}{2\pi} \int_0^{2\pi} \Sigma_t(\Phi_t(\theta)) \vec{H}_t(\Phi_t(\theta))\cdot \vec{V}_t(\Phi_t(\theta)) d\theta.\label{sigmahv}
\ee
Assuming that \(s=s_t\in[0,l_t]\) is a unit speed reparametrization of \(\Gamma_t\), then we have \(\partial_s=\frac{2\pi}{l_t}\partial_\theta\), and hence we have
\be
\vec{V}(\Phi_t(\theta))=\dot{\Phi}_t(\theta),\quad \vec{H}(\Phi_t(\theta))=\frac{4\pi^2}{l_t^2}\partial_\theta^2 \Phi_t(\theta),\,\theta\in [0,2\pi].\label{vandh}
\ee
Moreover, the differential equation for \(\bar{\Sigma}_t(\theta)\coloneqq\Sigma_t(\Phi_t(\theta))\) is
\be
\bca
\displaystyle
-\partial_\theta^2 \bar{\Sigma}_t = \frac{4\pi^2}{l_t^2}\lmb\frac{1}{2\pi}\int_0^{2\pi}|\partial_\theta^2\Phi_t|^2 d\theta - |\partial_\theta^2\Phi_t|^2\rmb,&\theta\in [0,2\pi],\\
\displaystyle\int_0^{2\pi} \bar{\Sigma}_t d\theta= 0.&
\eca\label{sigmaintheta}
\ee
We solve equation \eqref{sigmaintheta} via the Green's kernel \(G(\theta,\xi):[0,2\pi]^2\rta\R\) satisfying
\be
\bca
\displaystyle
-\partial_\theta^2 G(\theta,\xi) = \delta_\xi(\theta)-\frac{1}{2\pi},&\xi,\theta\in [0,2\pi],\\
\displaystyle\int_0^{2\pi} G(\theta,\xi) d\theta= 0.&
\eca\label{fundam}
\ee
The Green's kernel \(G\) to \eqref{fundam} has the following form, although we will not need it,
\be
G(\theta,\xi)=\frac{1}{4\pi}(\xi-\theta)^2+\min(\xi,\theta)-\frac{1}{2}(\xi+\theta)+\frac{1}{3}\pi.\label{fundamsol}
\ee
Now, the solution \(\bar{\Sigma}_t\) can be written as 
\be
\bar{\Sigma}_t(\theta)=-\frac{4\pi^2}{l_t^2}\int_0^{2\pi}G(\theta,\xi)|\partial_\theta^2\Phi_t(\xi)|^2 d\xi.\label{solutosigmaintheta}
\ee
Collecting \eqref{vandh} and \eqref{solutosigmaintheta}, we may rewrite \eqref{sigmahv} as
\be
\int \Sigma_t \vec{H}_t\cdot \vec{V}_t d\Gamma_t=-\frac{8\pi^3}{l_t^4}\int_0^{2\pi} \int_0^{2\pi} G(\theta,\xi)|\partial_\theta^2\Phi_t(\xi)|^2 \partial_\theta^2\Phi_t\cdot\dot{\Phi}_t(\theta) d\xi d\theta.
\ee
Introducing \(\Tilde{\Phi}_t\coloneqq \Phi_t/l_t\), we have, ignoring the constants
\be
\begin{split}
   \frac{1}{l_t^4}\int_0^{2\pi} \int_0^{2\pi} G(\theta,\xi)|\partial_\theta^2\Phi_t(\xi)|^2 \partial_\theta^2\Phi_t\cdot\dot{\Phi}_t(\theta) d\xi d\theta&= \int_0^{2\pi} \int_0^{2\pi} G(\theta,\xi)|\partial_\theta^2\Tilde{\Phi}_t(\xi)|^2 \partial_\theta^2\Tilde{\Phi}_t\cdot\dot{\Phi}_t(\theta)/l_t d\xi d\theta\\
   &=\int_0^{2\pi} \int_0^{2\pi} G(\theta,\xi)|\partial_\theta^2\Tilde{\Phi}_t(\xi)|^2 \partial_\theta^2\Tilde{\Phi}_t\cdot\dot{\Tilde{\Phi}}_t(\theta) d\xi d\theta\\
   &\quad+\partial_t\log{l_t}\int_0^{2\pi} \int_0^{2\pi} G(\theta,\xi)|\partial_\theta^2\Tilde{\Phi}_t(\xi)|^2 \partial_\theta^2\Tilde{\Phi}_t\cdot\Tilde{\Phi}_t(\theta) d\xi d\theta
\end{split}\label{goodidea}
\ee
Observe that \(\Tilde{\Phi}_t\) is also a constant speed parametrization of another closed path \(\Tilde{\Gamma}_t\coloneqq\text{Im}(\Tilde{\Phi}_t)\). Moreover, we have that \(\Ha^{d-1}(\Tilde{\Gamma}_t)\equiv 1\).

\begin{lem}
If Hypothesis \ref{Hypothesis} holds true for all closed paths, then given any smooth curve \(\Gamma\in \Ca_2\) and its unit speed parametrization \(\Phi:\mathbb{S}^1\rta \Gamma\subset\R^2\),
\be
\int_0^{2\pi}|\Phi|^2|D_\theta^2\Phi|^2 d\theta-\frac{1}{2\pi}\int_0^{2\pi}|\Phi|^2 d\theta\int_0^{2\pi}|D_\theta^2\Phi|^2 d\theta=0.\label{originalquantity}
\ee
In particular, for any \(\vec{u}\in \R^2\),
\be
\int_0^{2\pi}\Phi\cdot\vec{u}\,|D_\theta^2\Phi|^2 d\theta=\frac{1}{2\pi}\int_0^{2\pi}\Phi\cdot\vec{u}\, d\theta\int_0^{2\pi}|D_\theta^2\Phi|^2 d\theta.\label{translationvari}
\ee
\end{lem}

\begin{proof}
Integrating both sides of \eqref{goodidea} with respect to time \(t\in[0,1]\) we have, according to Hypothesis \ref{Hypothesis},
\be
\int_0^1\partial_t\log{l_t}\int_0^{2\pi} \int_0^{2\pi} G(\theta,\xi)|\partial_\theta^2\Tilde{\Phi}_t(\xi)|^2 \partial_\theta^2\Tilde{\Phi}_t\cdot\Tilde{\Phi}_t(\theta) d\xi d\theta dt =0.
\ee
Fixing \(\Tilde{\Phi}_t\), we observe that equation \eqref{goodidea} still holds if we replace \(\Phi_t\) by \(\Tilde{l}_t\Tilde{\Phi}_t\) for any smooth positive 1-periodic function \(\Tilde{l}_t\). This implies that 
\be
\int_0^{2\pi} \int_0^{2\pi} G(\theta,\xi)|\partial_\theta^2\Tilde{\Phi}_t(\xi)|^2 \partial_\theta^2\Tilde{\Phi}_t\cdot\Tilde{\Phi}_t(\theta) d\xi d\theta
\ee
is independent of time. Since by Remark \ref{remark1} (v) \(\Ca_2\) is path-connected, the above quantity should be a constant for all elements in \(\Ca_2\) that have perimeter 1. On the other hand, we have 
\[
\begin{split}
    \partial_\theta^2|\Tilde{\Phi}_t|^2&=2\partial_\theta^2\Tilde{\Phi}_t\cdot\Tilde{\Phi}_t + 2 |\partial_\theta\Tilde{\Phi}_t|^2\\
&=2\partial_\theta^2\Tilde{\Phi}_t\cdot\Tilde{\Phi}_t + 2,
\end{split}
\]
and then
\be
\begin{split}
 \int_0^{2\pi} \int_0^{2\pi} G(\theta,\xi)|\partial_\theta^2\Tilde{\Phi}_t(\xi)|^2 \partial_\theta^2\Tilde{\Phi}_t\cdot\Tilde{\Phi}_t(\theta) d\xi d\theta&=\frac{1}{2}\int_0^{2\pi}  |\partial_\theta^2\Tilde{\Phi}_t(\xi)|^2 \int_0^{2\pi}G(\theta,\xi)\partial_\theta^2|\Tilde{\Phi}_t|^2 d\theta d\xi \\
 &=-\frac{1}{2}\int_0^{2\pi}|\Tilde{\Phi}_t|^2|\partial_\theta^2\Tilde{\Phi}_t|^2 d\theta\\
 &\quad+\frac{1}{4\pi}\int_0^{2\pi}|\Tilde{\Phi}_t|^2 d\theta\int_0^{2\pi}|\partial_\theta^2\Tilde{\Phi}_t|^2 d\theta.
\end{split}\label{afterg}
\ee
We have shown \eqref{originalquantity} by observing that the quantity above is clearly 0 for any circle with perimeter 1. Equation \eqref{translationvari} is derived by variation of \eqref{originalquantity} in the direction \(\vec{u}\).
\end{proof}

\subsection{The Construction of a Counter-example}

Although \eqref{translationvari} seems unlikely to hold for all curves in \(\Ca_2\), we provide here a counter-example to exhibit precisely the contradiction.
\begin{lem}
There exists a smooth constant speed embedding \(\Phi:\mathbb{S}^1\rta\Gamma\subset\R^2\) such that there is a \(\vec{u}^\ast\in\R^2\)
\[
\int_0^{2\pi}\Phi\cdot\vec{u}^\ast\,|D_\theta^2\Phi|^2 d\theta\ne\frac{1}{2\pi}\int_0^{2\pi}\Phi\cdot\vec{u}^\ast\, d\theta\int_0^{2\pi}|D_\theta^2\Phi|^2 d\theta.
\]\label{counterexample1}
\end{lem}

\begin{proof}
We would like to construct an example of the shape illustrated in Figure \ref{counterpic}.
\begin{figure}[htbp]
         \centering
         \includegraphics[width=6.5cm]{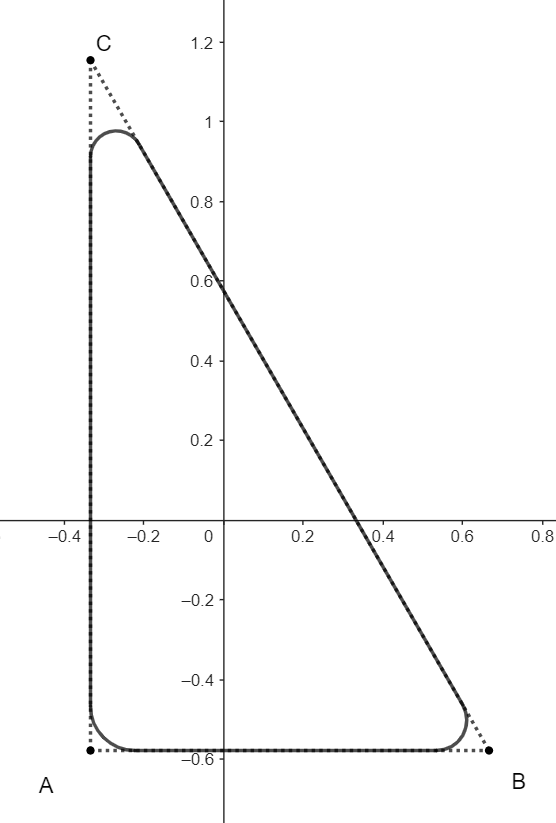}
        \caption{The example curve \(\Gamma\) is close to a right triangle with mass center 0; \(A=\lb-\frac{1}{3},-\frac{\sqrt{3}}{3}\rb,\,B=\lb\frac{2}{3},-\frac{\sqrt{3}}{3}\rb,\,C=\lb-\frac{1}{3},\frac{2\sqrt{3}}{3}\rb\).}
         \label{counterpic}
     \end{figure}
Let \(\Gamma_\varepsilon,\,0<\varepsilon\ll1\) denote a family of curves that are smooth in a neighborhood of scale \(O(\varepsilon)\) at each node \(A,B\) and \(C\), and coinciding with the triangle \(\Delta ABC\) elsewhere, and \(p_\varepsilon\) is defined as their arc-lengths. To illustrate the shape of \(\Gamma_\varepsilon\), we translate one of \(A,B,C\) to the origin and rotate the triangle so that the triangle can be locally written as the function graph of 
\be
\psi(x)=\bca
\cot({\alpha/2})x,&x\in[0,\varepsilon],\\
-\cot({\alpha/2})x,&x\in[-\varepsilon,0),
\eca\label{rotrform}
\ee
where \(\alpha\in[0,2\pi]\) denotes the open angle of the cone at the chosen node. We would like to construct \(\Gamma_\varepsilon\)'s by replacing \(\psi\) by some function \(u\) satisfying
\be
\bca
u''(x)=\phi_\varepsilon(x),&x\in[-\varepsilon,\varepsilon],\\
u'(-\varepsilon)=-u'(\varepsilon)=-\cot({\alpha/2}),&\\
u(-\varepsilon)=u(\varepsilon)=\cot({\alpha/2})\varepsilon,&
\eca\label{udoubleprime}
\ee
where we would like to choose \(\phi_\varepsilon\in C_0^\infty(-\varepsilon,\varepsilon)\) to be an even function satisfying \(\phi_\varepsilon(x)\equiv K>0\) on \(x\in(-\varepsilon+\varepsilon^2,\varepsilon-\varepsilon^2)\) and \(0\le\phi_\varepsilon\le K\) elsewhere. Observe that the replacement of \(\psi\) by \(u\) gives us a smooth curve near the chosen node. 

Before computing \eqref{translationvari}, let us compute some basic qualities of \(u\). By compatibility condition, we have 
\be
(2\varepsilon  + O(\varepsilon^2)) K \approx \int_{-\varepsilon}^{\varepsilon} \phi_\varepsilon(x) dx = 2 \cot({\alpha/2}),
\ee
which implies that 
\be
K= \frac{\cot({\alpha/2})}{\varepsilon} + O(1).
\ee
On the other hand, we have the derivative of \(u\) has the form
\be
\begin{split}
   u'(x)&=\int_{-\varepsilon}^x \phi_\varepsilon(y) dy - \frac{1}{2}\int_{-\varepsilon}^\varepsilon \phi_\varepsilon(z) dz\\
   &= K x + O(\varepsilon)\\
   &= \frac{\cot({\alpha/2})}{\varepsilon}x + O(\varepsilon).
\end{split}\label{uprime}
\ee
Collecting these values, we are now able to compute \eqref{translationvari}. Observe that because our refinement of triangle \(\Delta ABC\) is only at scale \(\varepsilon\), we have the following asymptotics (where \(\Phi_\varepsilon(s)\) is some unit speed reparametrization of \(\Gamma_\varepsilon\))
\be
p_\varepsilon\approx3+\sqrt{3}, \,\int\Phi_\varepsilon ds \int |D_s^2\Phi_\varepsilon|^2 ds = O( K^2 \varepsilon^2 )=O(1),\,\label{estimate1}
\ee
and
\be
    \int\Phi_\varepsilon |D_s^2\Phi_\varepsilon|^2 ds = \int_{\textbf{B}(A,10\varepsilon)\cap \Gamma_\varepsilon} \kappa^2 ds A + \int_{\textbf{B}(B,10\varepsilon)\cap \Gamma_\varepsilon} \kappa^2 ds B + \int_{\textbf{B}(C,10\varepsilon)\cap \Gamma_\varepsilon} \kappa^2 ds C+O(1).
\label{estimate2}
\ee
Using \eqref{uprime} and \eqref{udoubleprime}, we have 
\be
\begin{split}
  \int_{\textbf{B}(A,10\varepsilon)\cap \Gamma_\varepsilon} \kappa^2 ds &= \int_{-\varepsilon}^\varepsilon \frac{|u''|^2}{(1+|u'|^2)^3} \sqrt{1+|u'|^2} dx\\
  &=\int_{-\varepsilon}^\varepsilon \frac{|\phi_\varepsilon|^2}{(1+|u'|^2)^{5/2}} dx\\
  &=K^2\int_{-\varepsilon}^\varepsilon \frac{dx}{\lb1+\lb Kx\rb^2\rb^{5/2}}  + O(1)\\
  &=\lb\frac{\cot({\alpha_A/2})}{\varepsilon}\rb^2 \int_{-\varepsilon}^\varepsilon \frac{dx}{\lb1+\lb\frac{\cot({\alpha_A/2})}{\varepsilon}x\rb^2\rb^{5/2}} +O(1)\\
  &=\frac{\cot({\alpha_A/2})}{\varepsilon}\int_{-\cot({\alpha_A/2})}^{\cot({\alpha_A/2})} \frac{dx}{\lb1+x^2\rb^{5/2}} +O(1).
\end{split}
\ee
Plugging \(\alpha_A=\pi/2\), we have 
\be
\begin{split}
    \int_{\textbf{B}(A,10\varepsilon)} \kappa^2 ds &= \frac{\cot(\pi/4)}{\varepsilon}\int_{-\cot(\pi/4)}^{\cot(\pi/4)} \frac{dx}{\lb1+ x^2\rb^{5/2}}  + O(1)\\
    &=\frac{1}{\varepsilon}\int_{-1}^{1} \frac{dx}{\lb1+x^2\rb^{5/2}}  + O(1)\\
    &=\frac{5\sqrt{2}}{6}\cdot\frac{1}{\varepsilon} + O(1).
\end{split}\label{compA}
\ee
Similarly, we have 
\be
\begin{split}
    \int_{\textbf{B}(B,10\varepsilon)} \kappa^2 ds &= \frac{\cot({\alpha_B/2})}{\varepsilon}\int_{-\cot({\alpha_B/2})}^{\cot({\alpha_B/2})} \frac{1}{\lb1+ x^2\rb^{5/2}} dx + O(1)\\
    &=\frac{\cot(\pi/6)}{\varepsilon}\int_{-\cot(\pi/6)}^{\cot(\pi/6)} \frac{dx}{\lb1+x^2\rb^{5/2}}  + O(1)\\
    &=\frac{\sqrt{3}}{\varepsilon}\int_{-\sqrt{3}}^{\sqrt{3}} \frac{dx}{\lb1+x^2\rb^{5/2}}  + O(1)\\
    &=\frac{9}{4}\cdot\frac{1}{\varepsilon}+O(1),
\end{split}\label{compB}
\ee
and 
\be
\begin{split}
    \int_{\textbf{B}(C,10\varepsilon)} \kappa^2 ds &= \frac{\cot({\alpha_C/2})}{\varepsilon}\int_{-\cot({\alpha_C/2})}^{\cot({\alpha_C/2})} \frac{1}{\lb1+x^2\rb^{5/2}} dx + O(1)\\
    &=\frac{\cot(\pi/12)}{\varepsilon}\int_{-\cot(\pi/12)}^{\cot(\pi/12)} \frac{dx}{\lb1+x^2\rb^{5/2}}  + O(1)\\
    &=\frac{2+\sqrt{3}}{\varepsilon}\int_{-\sqrt{3}-2}^{\sqrt{3}+2} \frac{dx}{\lb1+x^2\rb^{5/2}}  + O(1)\\
    &=\frac{41\sqrt{2}+25\sqrt{6}}{24}\cdot\frac{1}{\varepsilon} + O(1).
\end{split}\label{compC}
\ee
Combining \eqref{estimate1}, \eqref{estimate2}, \eqref{compA}, \eqref{compB} and \eqref{compC}, we may make the conclusion that when \(\varepsilon>0\) is very small,
\be
\begin{split}
   \int\Phi_\varepsilon\,|D_s^2\Phi_\varepsilon|^2 ds-\frac{1}{p_\varepsilon}\int\Phi_\varepsilon\, ds\int|D_s^2\Phi_\varepsilon|^2 ds&=\frac{1}{\varepsilon}\lb\frac{5\sqrt{2}}{6}A+\frac{9}{4}B+\frac{41\sqrt{2}+25\sqrt{6}}{24}C\rb+O(1)\\
&=\frac{1}{\varepsilon}\lb-\frac{1}{3}\cdot\frac{5\sqrt{2}}{6}+ \frac{2}{3}\cdot\frac{9}{4}-\frac{1}{3}\cdot\frac{41\sqrt{2}+25\sqrt{6}}{24},\star\rb+O(1)\\
&\approx\frac{1}{\varepsilon}\lb0.5487,\star\rb\ne0.\\
\end{split}
\ee
The proof is done by choosing \(\Phi(\theta)=\Phi_\varepsilon\lb\frac{p_\varepsilon\theta}{2\pi}\rb\) for a small \(\varepsilon>0\).

\end{proof}

\subsection{The Proof of the Second Part of Theorem \ref{2}}

Similar to Section 3, we prove the theorem by giving a contradiction to the following hypothesis:
\begin{hyp}
For all closed paths \(\Gamma_t\) in the space of plane curves driven by a vector field \(\vec{V}_t=\vec{V}_t^\bot\),
\be
\int_0^1\int_{\Gamma_t}\lb1+|\vec{H}_t|^2\rb \vec{H}_t\cdot\vec{V}_t\, d{\Ha^{1}} dt = 0.\label{hyp3.28}
\ee\label{Hypothesis2}
\end{hyp}
Observe that the first term on the left-hand side of Hypothesis \eqref{hyp3.28} is 0 by using the first variation formula of perimeters. By using \eqref{vandh} we can rewrite the second term of the left-hand side as (where \(C\) is a computable constant)
\be
\int_0^1\int_{\Gamma_t}|\vec{H}_t|^2 \vec{H}_t\cdot\vec{V}_t\, d{\Ha^{1}} dt=C\int_0^1 \frac{1}{l_t^5} \int_0^{2\pi} |\partial_\theta^2\Phi_t(\theta)|^2 \partial_\theta^2\Phi_t(\theta)\cdot \dot{\Phi}_t(\theta) d\theta dt.
\ee
Introducing \(\Tilde{\Phi}_t=\Phi_t/l_t\), we have
\be
\begin{split}
  \frac{1}{l_t^5} \int_0^{2\pi} |\partial_\theta^2\Phi_t(\theta)|^2 \partial_\theta^2\Phi_t(\theta)\cdot \dot{{\Phi}}_t(\theta)  d\theta &= \frac{1}{l_t}\int_0^{2\pi} |\partial_\theta^2\Tilde{\Phi}_t(\theta)|^2 \partial_\theta^2\Tilde{\Phi}_t(\theta)\cdot \dot{\Tilde{\Phi}}_t(\theta)  d\theta \\
  &\quad-\partial_t \lb\frac{1}{l_t}\rb\int_0^{2\pi} |\partial_\theta^2\Tilde{\Phi}_t(\theta)|^2 \partial_\theta^2\Tilde{\Phi}_t(\theta)\cdot \Tilde{\Phi}_t(\theta) d\theta.
\end{split}\label{secondgoodfor}
\ee

\begin{lem}
If Hypothesis \ref{Hypothesis2} holds true for all closed paths in the space of plane curves, then given any smooth curve \(\Gamma\in \Ca_2\) and its unit speed parametrization \(\Phi:\mathbb{S}^1\rta \Gamma\subset\R^2\), we have
\be
\int_0^{2\pi} |D_\theta^2\Phi|^2 D_\theta^2\Phi d\theta =0.\label{secondcriterion}
\ee
\end{lem}

\begin{proof}
Integrating both sides of \eqref{secondgoodfor} with respect to time \(t\in[0,1]\) we obtain by Hypothesis \ref{Hypothesis2}
\be
\int_0^1\frac{1}{l_t}\int_0^{2\pi} |\partial_\theta^2\Tilde{\Phi}_t(\theta)|^2 \partial_\theta^2\Tilde{\Phi}_t(\theta)\cdot \dot{\Tilde{\Phi}}_t(\theta)  d\theta -\partial_t \lb\frac{1}{l_t}\rb\int_0^{2\pi} |\partial_\theta^2\Tilde{\Phi}_t(\theta)|^2 \partial_\theta^2\Tilde{\Phi}_t(\theta)\cdot \Tilde{\Phi}_t(\theta) d\theta dt = 0.\label{intebypart}
\ee
Fixing \(\Tilde{\Phi}_t\), and replacing \(\Phi_t\) in \eqref{secondgoodfor} by \(\Tilde{l}_t \Tilde{\Phi}_t\) for an arbitrary smooth 1-periodic function \(\Tilde{l}_t>0\), we observe by integration by parts in \(t\) in \eqref{intebypart} that \(\Tilde{\Phi}_t\) should satisfy
\be
\partial_t \int_0^{2\pi} |\partial_\theta^2\Tilde{\Phi}_t|^2 \partial_\theta^2\Tilde{\Phi}_t\cdot \Tilde{\Phi}_t(\theta) d\theta + \int_0^{2\pi} |\partial_\theta^2\Tilde{\Phi}_t|^2 \partial_\theta^2\Tilde{\Phi}_t\cdot \dot{\Tilde{\Phi}}_t(\theta)  d\theta=0.
\ee
Replacing \(\Tilde{\Phi}_t\) above by \(\Tilde{\Phi}_t+\vec{u}\) for some \(\vec{u}\in\R^2\), we see that \(\Tilde{\Phi}_t\) should satisfy
\be
\partial_t\int_0^{2\pi} |\partial_\theta\Tilde{\Phi}_t|^2 \partial_\theta^2\Tilde{\Phi}_t d\theta = 0.
\ee
Similar to Section 3, path-connectivity of \(\Ca_2\) and the computation about circles imply \eqref{secondcriterion}.

\end{proof}

Now let us construct a counterexample to \eqref{secondcriterion}.

\begin{lem}
There exists a smooth constant speed embedding \(\Phi:\mathbb{S}^1\rta\Gamma\subset\R^2\) such that
\be
\int_0^{2\pi} |D_\theta^2\Phi|^2 D_\theta^2\Phi d\theta \ne0.\label{countersecondcriterion}
\ee
\end{lem}

\begin{proof}

We would like to construct a curve of the shape in Figure \ref{counter2}.
\begin{figure}[htbp]
        \centering
        \includegraphics[width=14cm]{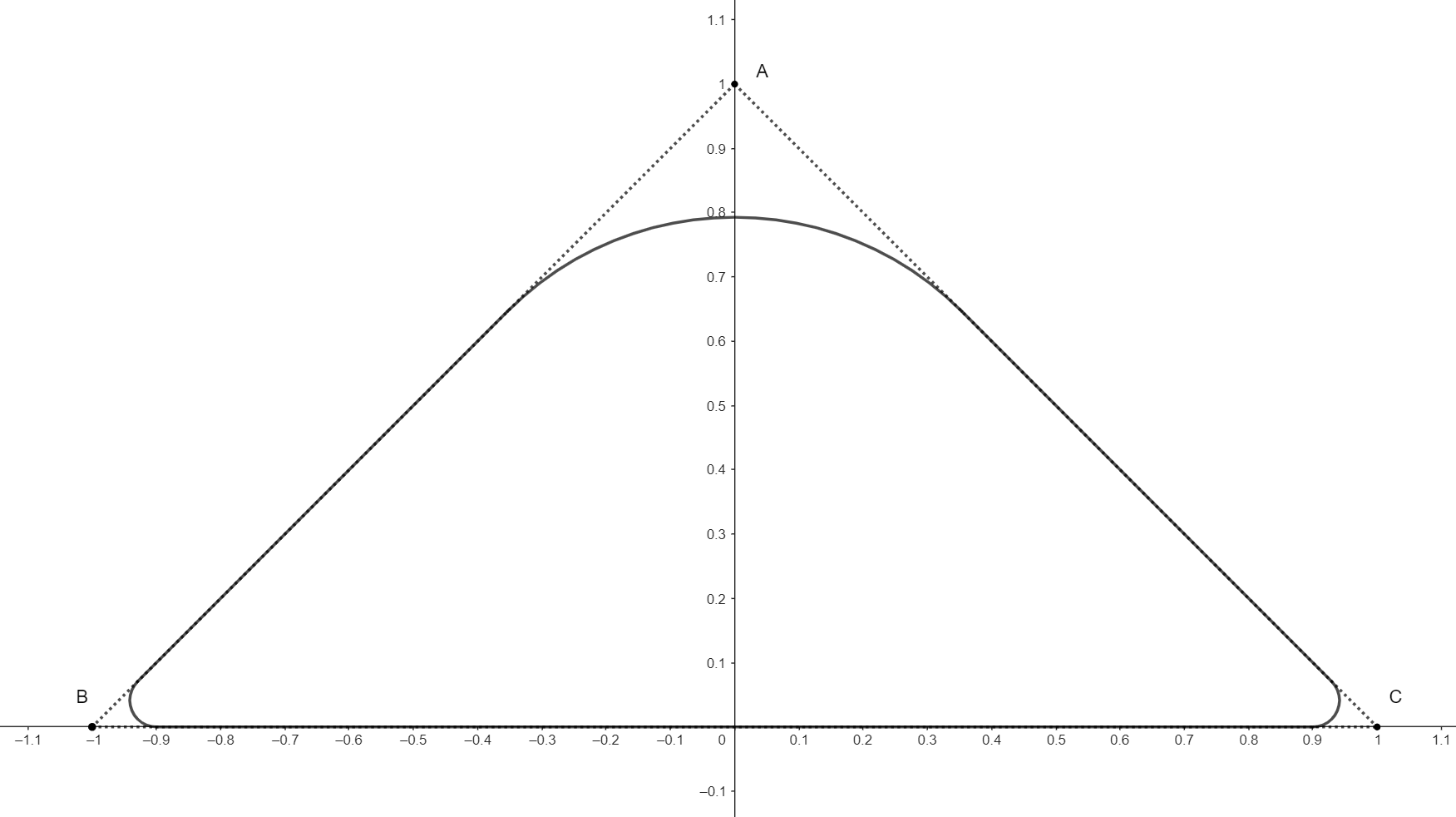}
        \caption{The curve is close to the right triangle \(\Delta ABC\), with \(O(1)\) mollification at \(A=(0,1)\), and \(O(\varepsilon)\) mollification at \(B=(-1,0)\) and \(C=(1,0)\) for some \(\varepsilon>0\) small.}
        \label{counter2}
    \end{figure}

 Similar to the proof of lemma \eqref{counterexample1}, we mollify the conic points \(A,B\) and \(C\) by locally replacing the curve by the function graph of the solution to equation \eqref{udoubleprime}. Near point \(A\), the mollification is at scale \(O(1)\), but near \(B\) and \(C\) the mollifications are at scale \(O(\varepsilon)\). Let \(\Gamma_\varepsilon\) denote the family of the mollified curves, \( 4.2022\lessapprox p_\varepsilon < 2+2\sqrt{2}\) be their arc-lengths and \(\Phi^\varepsilon(s):[0,p_\varepsilon]\rta\Gamma_\varepsilon\subset\R^2\) be a family of unit speed reparametrization of the mollified curves at scale \(\varepsilon\), we have the following asymptotics
\be
\begin{split}
   \int |D_s^2\Phi^\varepsilon|^2 D_s^2\Phi^\varepsilon ds
&= \int_{\textbf{B}(B,10\varepsilon)\cap \Gamma_\varepsilon} \kappa^3 \vec{N} d s + \int_{\textbf{B}(C,10\varepsilon)\cap \Gamma_\varepsilon} \kappa^3 \vec{N} d s + O(1)\\
 &\eqqcolon \Tilde{B}+\Tilde{C} +O(1).
\end{split}
\ee
Because the curve is symmetric with respect to the vertical axis, we observe that \(\Tilde{B}\) and \(\Tilde{C}\) share the same vertical component and \(\Tilde{B}+\Tilde{C}\) has zero horizontal component. Moreover, since the mollified curve near \(C\) (or \(B\)) is locally symmetric with respect to the middle-angle line passing through \(C\) (or \(B\)), the vector \(\Tilde{C}\) has a fixed direction that is not parallel to the horizontal line, and hence it suffices to show that 
\be
\lw\int_{\textbf{B}(C,10\varepsilon)\cap \Gamma_\varepsilon} \kappa^3 \vec{N} d s\rw\gg O(1).
\ee
In fact, after rotation and translation of the curve so that \(C\) is at the origin and the curve is locally of the form \eqref{rotrform}, we have
\be
\begin{split}
   \lw\int_{\textbf{B}(C,10\varepsilon)\cap \Gamma_\varepsilon} \kappa^3 \vec{N} d s\rw & = \int_{-\varepsilon}^\varepsilon \frac{(u'')^3}{(1+(u')^2)^{9/2}} dx\\
   &\approx \lb\frac{\cot(\alpha_C/2)}{\varepsilon}\rb^2 \int_{-\cot(\alpha_C/2)}^{\cot(\alpha_C/2)} \frac{dx}{\lb1+x^2\rb^{9/2}}   \\
   &\gg O(1).
\end{split}
\ee
The proof is done by choosing \(\Phi(\theta)=\Phi^\varepsilon\lb\frac{p_\varepsilon\theta}{2\pi}\rb\) for a small \(\varepsilon>0\).

\end{proof}

\section*{Acknowledgements}
I would like to express my deepest appreciation to my supervisor William Feldman for his patient guidance and profound belief in my work. I also thank Andrejs Treibergs and Dohyun Kwon for the helpful discussions and comments on this research.

\bibliographystyle{plainnat}
\bibliography{ref}

\end{document}